\documentclass[journal,onecolumn]{IEEEtran}

\usepackage{mathrsfs}
\usepackage{color}
\usepackage{enumerate}
\usepackage{mathrsfs}
\usepackage{amssymb}
\usepackage{amsfonts}
\usepackage{amsmath}
\usepackage{multirow}
\usepackage{amsthm}
\usepackage{url}

\newtheorem{theorem}{Theorem}[section]

\newtheorem{lemma}[theorem]{Lemma}

\newtheorem{proposition}[theorem]{Proposition}

\begin{document}

\title{On lattice tilings of $\mathbb{Z}^{n}$ by limited magnitude error balls $\mathcal{B}(n,2,1,1)$}

\author{Tao Zhang, Yanlu Lian and Gennian Ge
\thanks{The research of G. Ge was supported by the National Key Research and Development Program of China under Grant 2020YFA0712100 and Grant 2018YFA0704703, the National Natural Science Foundation of China under Grant 11971325 and Grant 12231014, and Beijing Scholars Program.}
\thanks{T. Zhang is with the Zhejiang Lab, Hangzhou 311100, China (e-mail: zhant220@163.com).}
\thanks{Y. Lian is with the School of Mathematics, Hangzhou Normal University, Hangzhou 311100, China (e-mail: yllian@hznu.edu.cn).}
\thanks{G. Ge is with the School of Mathematics Sciences, Capital Normal University, Beijing 100048, China (e-mail: gnge@zju.edu.cn).}
}

\maketitle
\begin{abstract}
Limited magnitude error model has applications in flash memory. In this model, a perfect code is equivalent to a tiling of $\mathbb{Z}^n$ by limited magnitude error balls. In this paper, we give a complete classification of lattice tilings of $\mathbb{Z}^n$ by limited magnitude error balls $\mathcal{B}(n,2,1,1)$.
\end{abstract}

\begin{IEEEkeywords}
Lattice tiling, limited magnitude errors, flash memory.
\end{IEEEkeywords}
\section{Introduction}
In the asymmetric limited magnitude error model, a symbol $a\in\mathbb{Z}$ may be modified to $b$ during transmission, and the error magnitude $|a-b|$ is likely to be bounded by certain threshold. One of applications of asymmetric limited magnitude error model is flash memory \cite{CSBB10}. Moreover, for an information $\bf{a}\in\mathbb{Z}^{n}$, a common noise affects only some of the entries. Hence, for integers $n\ge t\ge1$ and $k_{+}\ge k_{-}\ge0$, we define the $(n,t,k_{+},k_{-})$-error ball as
\[\mathcal{B}(n,t,k_{+},k_{-}):=\{{\bf{a}}=(a_1,a_2,\dots,a_n)\in\mathbb{Z}^n:\ a_i\in[-k_{-},k_{+}],\ \text{wt}({\bf{a}})\le t\},\]
where $\text{wt}({\bf{a}})$ denotes the Hamming weight of ${\bf{a}}$. Under this setting, it is easy to see that an error correcting code is equivalent to a packing of $\mathbb{Z}^n$ by $\mathcal{B}(n,t,k_{+},k_{-})$, and a perfect code is equivalent to a tiling of $\mathbb{Z}^{n}$ by $\mathcal{B}(n,t,k_{+},k_{-})$. Moreover, a linear perfect code is equivalent to a lattice tiling of $\mathbb{Z}^{n}$ by $\mathcal{B}(n,t,k_{+},k_{-})$.

When $n=t$, it is easy to see that $\mathcal{B}(n,t,k_{+},k_{-})$ is a cube $[-k_{+}-k_{-},k_{+}+k_{-}]^n$, hence we will only consider $1\le t\le n-1$.
Tilings (or packings) of $\mathbb{Z}^{n}$ by $\mathcal{B}(n,1,k_{+},k_{-})$ have been intensively studied in recent years due to their own interests and applications in flash memory. See \cite{HS86,KBE11,KLNY11,KLY12,S67,S84,SS94,S86,S87,T98,W95} for the researches on cross $\mathcal{B}(n,1,k,k)$ and semi-cross $\mathcal{B}(n,1,k,0)$. Later, these two shapes are extended to quasi-cross $\mathcal{B}(n,1,k_{+},k_{-})$ by Schwartz \cite{S12}, and immediately, quasi-cross was received a lot of attentions \cite{S14,XL20,XL21,YKB13,YZZG20,YZ20,ZG16,ZG18,ZZG17}. For $t\ge2$, there are only a few results. Tilings (or packings) of $\mathbb{Z}^{n}$ by $\mathcal{B}(n,n-1,k,0)$ are considered in \cite{BE13,KLNY11,S90}. In \cite{WS22,WWS21}, the authors studied the tiling and packing problem in the general case. In particular, Wei and Schwartz \cite{WS22} gave a complete classification of the lattice tilings with $\mathcal{B}(n,2,1,0)$ and $\mathcal{B}(n,2,2,0)$. Recently, Wei and Schwartz \cite{WS2023} considered perfect burst-correcting codes for the limited magnitude error channel.

The goal of this paper is to continue this research. Since the lattice tilings with $\mathcal{B}(n,2,1,0)$ and $\mathcal{B}(n,2,2,0)$ have been completely classified, and $\mathcal{B}(n,2,1,1)$ is the next open case. In this paper, we will solve this case. The main result of this paper is the following one.
\begin{theorem}\label{mainthm}
For $n\ge3$, there exists a lattice tiling of $\mathbb{Z}^n$ by $\mathcal{B}(n,2,1,1)$ if and only if $n=11$.
\end{theorem}
This paper is organized as follows. In Section~\ref{pre}, we will give the group ring representation of lattice tilings of $\mathbb{Z}^n$ by $\mathcal{B}(n,2,1,1)$. In Section~\ref{mainsec}, we prove our main result.
\section{Preliminaries}\label{pre}
Let $\mathbb{Z}[G]$ denote the group ring of $G$ over $\mathbb{Z}$, where $G$ is a finite abelian group (written multiplicatively). For any $A\in\mathbb{Z}[G]$, $A$ can be written as formal sum $A=\sum_{g\in G}a_gg$, where $a_g\in\mathbb{Z}$. For any $A=\sum_{g\in G}a_gg$ and $t\in\mathbb{Z}$, we define
\[A^{(t)}=\sum_{g\in G}a_gg^t.\]
 Addition, subtraction, multiplication and scalar multiplication in group ring are defined as:
\[\sum_{g\in G}a_gg\pm\sum_{g\in G}b_gg=\sum_{g\in G}(a_g\pm b_g)g,\]
\[\sum_{g\in G}a_gg\sum_{g\in G}b_gg=\sum_{g\in G}(\sum_{h\in G}a_{h}b_{h^{-1}g})g,\]
and
\[\lambda\sum_{g\in G}a_gg=\sum_{g\in G}(\lambda a_g)g,\]
where $\lambda\in\mathbb{Z}$ and $\sum_{g\in G}a_gg,\sum_{g\in G}b_gg\in\mathbb{Z}[G]$. For any multiset $A$ of $G$, we can identify $A$ with the group ring element $\sum_{g\in G}a_gg$, where $a_g$ is the multiplicity of $g$ appearing in $A$. We also define $\text{supp}(A)$ to be the support of $A=\sum_{g\in G}a_gg$, i.e., $\text{supp}(A)=\{g\in G: a_g\ne0\}$. For any $A=\sum_{g\in G}a_gg\in\mathbb{Z}[G]$, we denote
\[A^{*}=\sum_{g\in G\backslash\{e\}}a_gg.\]

The following theorem establishes the connection between lattice tilings of $\mathbb{Z}^n$ and finite abelian groups.
\begin{theorem}{\rm{\cite{HA12}}}\label{tilinggroup}
Let $V$ be a subset of $\mathbb{Z}^{n}$. Then there is a lattice tiling $\mathcal{T}$ of $\mathbb{Z}^n$ by $V$ if and only if there are both an abelian group $G$ of order $|V|$ and a homomorphism $\phi:\mathbb{Z}^n\rightarrow G$ such that the restriction of $\phi$ to $V$ is a bijection.
\end{theorem}
Now we translate the existence of lattice tilings of $\mathbb{Z}^n$ by $\mathcal{B}(n,2,1,1)$ into group ring equations.
\begin{theorem}\label{groupring}
Let $n\ge3$, then there exists a lattice tiling of $\mathbb{Z}^n$ by $\mathcal{B}(n,2,1,1)$ if and only if there exists a finite abelian group $G$ of order $2n^2+1$ and a subset $T\subseteq G$ satisfying
\begin{enumerate}
  \item [(1)] $|T|=2n+1$,
  \item [(2)] $e\in T$,
  \item [(3)] $T=T^{(-1)}$,
  \item [(4)] $T^2=2G+T^{(2)}+(2n-2)e$.
\end{enumerate}
\end{theorem}
\begin{proof}
Suppose that there exists a lattice tiling of $\mathbb{Z}^n$ by $\mathcal{B}(n,2,1,1)$.
Let $e_i$, $i=1,2,\dots,n$, be a fixed orthonormal basis of $\mathbb{Z}^{n}$.
  By Theorem~\ref{tilinggroup}, there are both an abelian group $G$ of order $2n^2+1$ and a homomorphism $\phi:\mathbb{Z}^n\rightarrow G$ such that the restriction of $\phi$ to $\mathcal{B}(n,2,1,1)$ is a bijection. Since the homomorphism $\phi$ is determined by the values $\phi(e_i)$, $i=1,\dots,n$, then there exists an $n$-subset $\{a_1,a_2,\dots,a_n\}\subset G$ (let $\phi(e_i)=a_i$) such that
  \begin{align*}
  G=\{e\}\cup\{a_i,a_{i}^{-1}:\ 1\le i\le n\}\cup\{a_ia_j,a_ia_{j}^{-1},a_{i}^{-1}a_j,a_{i}^{-1}a_{j}^{-1}:\ 1\le i<j\le n\}.
  \end{align*}
  In the language of group ring, the above equation can be written as
  \begin{align*}
  G=e+\sum_{i=1}^{n}(a_i+a_{i}^{-1})+\sum_{1\le i<j\le n}(a_i+a_{i}^{-1})(a_j+a_{j}^{-1}).
  \end{align*}
  Let $T=e+\sum_{i=1}^{n}(a_i+a_{i}^{-1})$, then $|T|=2n+1$, $e\in T$ and $T^{(-1)}=T$. Moreover,
  \begin{align*}
  T^2=&(e+\sum_{i=1}^{n}(a_i+a_{i}^{-1}))^2\\
  =&e+2\sum_{i=1}^{n}(a_i+a_{i}^{-1})+\sum_{i=1}^{n}(a_{i}^{2}+a_{i}^{-2})+2\sum_{1\le i<j\le n}(a_i+a_{i}^{-1})(a_j+a_{j}^{-1})+2ne\\
  =&2G+T^{(2)}+(2n-2)e.\qedhere
  \end{align*}
\end{proof}
\section{Proof of the main theorem}\label{mainsec}
In this section, we will prove our main result. The main method was developed by Leung and Zhou \cite{LZ20}, which completely solved the existence of linear perfect Lee codes with radius 2. This method was also used to determine the existence of almost perfect linear Lee codes of radius 2 \cite{XZ2022,ZZ2022} and linear diameter perfect Lee codes with diameter 6 \cite{ZG2022}.

The sketch of our proof is: we will first investigate $T^{(2)}T\pmod{3}$, we can get some restrictions on the coefficients of elements of $T^{(2)}T$. This will solve the cases $n\equiv2\pmod{3}$ except $n=5,11$. By studying $TT^{(3)}\pmod{3}$, we solve the case $n\equiv1\pmod{3}$ except $n=4$.
For the case $n\equiv0\pmod{3}$, we will also study $TT^{(4)}\pmod{5}$ and $TT^{(5)}\pmod{5}$. Small dimensions will be solved by symmetric polynomial method, which was developed by Kim \cite{K17}.

Let $n\ge3$ be an integer. Suppose that $G$ is a finite abelian group with order $2n^2+1$ and $T\subset G$ satisfying conditions in Theorem~\ref{groupring}. Recall that $T=e+\sum_{i=1}^{n}(a_i+a_{i}^{-1})$ and
\begin{align}\label{eq1}
 G=e+\sum_{i=1}^{n}(a_i+a_{i}^{-1})+\sum_{1\le i<j\le n}(a_i+a_{i}^{-1})(a_j+a_{j}^{-1}).
\end{align}
The following lemma is directly from Condition (4) of Theorem~\ref{groupring}.
  \begin{lemma}\label{lemma1}
  For any $t\ne e$, we have
 \[|\{(t_1,t_2)\in T\times T: t_1t_2=t\}|= \begin{cases}
    3, & \mbox{if } t\in T^{(2)}; \\
    2, & \mbox{if }t\notin T^{(2)}.
  \end{cases}\]
\end{lemma}
From Condition (4) in Theorem~\ref{groupring}, we have
\begin{align}\label{eq4}
TT^{(2)}=T^{3}-(4n+2)G-(2n-2)T.
\end{align}
We write
\begin{align}\label{eq2}
TT^{(2)}=\sum_{i=0}^{M}iX_i,
\end{align}
where $X_i$ $(i=0,1,\dots,M)$ form a partition of group $G$, i.e.,
\begin{align}\label{eq3}
G=\sum_{i=0}^{M}X_i.
\end{align}
Then we have the following lemma.
\begin{lemma}\label{lemma2}
  \begin{enumerate}
    \item [(1)] $\sum_{i=1}^{M}i|X_i|=4n^2+4n+1$,
    \item [(2)] $\sum_{i=0}^{M}|X_i|=2n^2+1$,
    \item [(3)] $\sum_{i=1}^{M}|X_i|=1-\beta+\sum_{i\ge3}\frac{(i-1)(i-2)}{2}|X_i|$, where $2\beta=|T^{*}\cap T^{(2)*}|$.
  \end{enumerate}
\end{lemma}
\begin{proof}
  The first two equations are directly from Equations~(\ref{eq2}) and (\ref{eq3}). Now we count the size of $\text{supp}(T^{(2)}T)$.
We write $T^{(2)}=\sum_{i=0}^{2n}b_i$, where $b_0=e$ and $b_{i}^{-1}=b_{n+i}$ for $i=1,2,\dots,n$.
  By Lemma~\ref{lemma1}, we have
   \[|b_iT\cap b_jT|= \begin{cases}
    3, & \mbox{if } b_ib_{j}^{-1}\in T^{(2)}; \\
    2, & \mbox{if }b_ib_{j}^{-1}\notin T^{(2)}.
  \end{cases}\]
  Let $\alpha=|\{(i,j): 0\le i<j\le 2n, b_ib_{j}^{-1}\in T^{(2)}\}|$ and $N(t)=|\{(i,j): 0\le i<j\le 2n, b_ib_{j}^{-1}=t\}|$, then $\alpha=\sum_{t\in T^{(2)}}N(t)$. Since $\gcd(2n^2+1,2)=1$, then $T^{(2)}$ also satisfies conditions in Theorem~\ref{groupring}. By the definition of $b_i$ and Lemma~\ref{lemma1}, we have
     \[N(t)+N(t^{-1})= \begin{cases}
    3, & \mbox{if } t\in T^{(2)*}\cap T^{(4)*}; \\
    2, & \mbox{if } t\in T^{(2)}\backslash T^{(4)}.
  \end{cases}\]
  Hence, $\alpha=|T^{(2)}\backslash T^{(4)}|+\frac{3|T^{(2)*}\cap T^{(4)*}|}{2}=|T^{*}\backslash T^{(2)*}|+\frac{3|T^{*}\cap T^{(2)*}|}{2}=2n+\beta$.

By the inclusion-exclusion principle, we have
  \begin{align*}
  \sum_{i=1}^{M}|X_i|=&|\text{supp}(T^{(2)}T)|\\
  =&\sum_{i=0}^{2n}|b_iT|-\sum_{0\le i<j\le 2n}|b_iT\cap b_jT|+\sum_{r\ge3}\sum_{0\le i_1<\dots<i_r\le 2n}(-1)^{r-1}|b_{i_1}T\cap\dots\cap b_{i_r}T|\\
  =&4n^2+4n+1-3\alpha-2(\binom{2n+1}{2}-\alpha)+\sum_{r\ge3}\sum_{0\le i_1<\dots<i_r\le 2n}(-1)^{r-1}|b_{i_1}T\cap\dots\cap b_{i_r}T|.
  \end{align*}
  Suppose that $g\in b_{i_1}T\cap\dots\cap b_{i_r}T$ for some $r\ge 3$, then $g\in X_i$ for some $i\ge3$. The contribution for $g$ in the sum $\sum_{r\ge3}\sum_{1\le i_1<\dots<i_r\le 2n}(-1)^{r-1}|b_{i_1}T\cap\dots\cap b_{i_r}T|$ is
  \[(-1)^{3-1}\binom{i}{3}+(-1)^{4-1}\binom{i}{4}+\cdots=\binom{i}{2}-\binom{i}{1}+\binom{i}{0}=\frac{(i-1)(i-2)}{2}.\]
  Hence, we have
  \begin{align*}
  \sum_{i=1}^{M}|X_i|=&4n^2+4n+1-3\alpha-2(\binom{2n+1}{2}-\alpha)+\sum_{r\ge3}\sum_{1\le i_1<\dots<i_r\le 2n}(-1)^{r-1}|b_{i_1}T\cap\dots\cap b_{i_r}T|\\
  =&4n^2+4n+1-3\alpha-2(\binom{2n+1}{2}-\alpha)+\sum_{i\ge3}|X_i|\frac{(i-1)(i-2)}{2}\\
  =&1-\beta+\sum_{i\ge3}\frac{(i-1)(i-2)}{2}|X_i|.\qedhere
  \end{align*}
\end{proof}

\begin{lemma}\label{lemma6}
$e$ appears $2\beta+1$ times in $T^{(3)}$, where $2\beta=|T^{*}\cap T^{(2)*}|$.
\end{lemma}
\begin{proof}
If there exist $a_i,a_j\in T^{*}$ with $j\ne i$ such that $a_i=a_{j}^{2}\in T^{*}\cap T^{(2)*}$, then $a_{j}=a_{i}a_{j}^{-1}$. By Equation~(\ref{eq1}), it can be seen that $a_{j}$ appears twice in $G$, which is a contradiction.

 If there exists $a_{i}\in T^{*}$ such that $a_i=a_{i}^{2}\in T^{*}\cap T^{(2)*}$, then $a_{i}=e$, which is a contradiction again.

 If there exists $a_{i}\in T^{*}$ such that $a_i=a_{i}^{-2}\in T^{*}\cap T^{(2)*}$, then $a_{i}^{3}=e$.

 On the other hand, if there exists $a_{i}\in T^{*}$ such that $a_{i}^3=e$, then $a_i=a_{i}^{-2}\in T^{*}\cap T^{(2)*}$. Hence $e$ appears $2\beta+1$ times in $T^{(3)}$.
\end{proof}

\subsection{$n\equiv2\pmod{3}$}
\begin{proposition}\label{prop1}
  Theorem~\ref{mainthm} holds for $n\equiv2\pmod{3}$ except for $n=5,11$.
\end{proposition}
\begin{proof}
By Equation~(\ref{eq4}), we obtain
\begin{align}\label{eq23}
T^{(2)}T\equiv T^{(3)}+2G+T\pmod{3}.
\end{align}
This implies that
\begin{align*}
&\sum_{i\ge0}|X_{3i}|=a_1,\\
&\sum_{i\ge0}|X_{3i+1}|=a_2,\\
&\sum_{i\ge0}|X_{3i+2}|=2n^2+1-a_1-a_2,
\end{align*}
where
\begin{align*}
&a_1=|T\backslash T^{(3)}|+b_1+b_2,\\
&a_2=b_3+b_4,\\
&b_1=|\{t\in T\cap T^{(3)}: t\text{ appears }3i\text{ times in }T^{(3)}\text{ for some }i\}|,\\
&b_2=|\{t\in T^{(3)}\backslash T: t\text{ appears }3i+1\text{ times in }T^{(3)}\text{ for some }i\}|,\\
&b_3=|\{t\in T\cap T^{(3)}: t\text{ appears }3i+1\text{ times in }T^{(3)}\text{ for some }i\}|,\\
&b_4=|\{t\in T^{(3)}\backslash T: t\text{ appears }3i+2\text{ times in }T^{(3)}\text{ for some }i\}|.
\end{align*}
Hence $a_1+a_2\le4n+1$. Since for any $t\in T^{*}$, $t=t^{-1}\cdot t^2$, $t^3=t\cdot t^2$, then Equation (\ref{eq23}) implies that for any $e\ne g\in G$, $g$ appears at least twice in $TT^{(3)}$. Hence $|X_0|=0$ and $|X_1|\le1$.

{\bf{Claim: $|X_1|=0$.}}

Assume to the contrary, $|X_1|=1$, then $e\in X_1$ and $\beta=0$. By Lemma~\ref{lemma2}, we have
\begin{equation}\label{eq7}
\begin{split}
\sum_{i=3}^{M}(i-2)|X_i|=&\sum_{i=1}^{M}i|X_i|-2\sum_{i=1}^{M}|X_i|+|X_1|\\
=&4n^2+4n+1-2(2n^2+1)+1\\
=&4n,
\end{split}
\end{equation}
and
\begin{equation}\label{eq8}
\begin{split}
2n^2=&\sum_{i=3}^{M}\frac{(i-1)(i-2)}{2}|X_i|\\
\le&\frac{M-1}{2}\sum_{i=3}^{M}(i-2)|X_i|\\
=&\frac{M-1}{2}4n.
\end{split}
\end{equation}
This leads to $M\ge n+1$. Note that for any $t\in T^{*}$, $t$ appears at least $3$ times in $TT^{(2)}$. Then Equation (\ref{eq7}) implies that
\begin{align*}
4n=&\sum_{i=3}^{M}(i-2)|X_i|\\
\ge&(2n-|X_M|)+(M-2)|X_M|.
\end{align*}
This leads to $(M-3)|X_M|\le2n$. So $|X_M|\le2$. Since $X_{M}=X_{M}^{(-1)}$, then $|X_M|=2$ and $M\le n+3$. Note that for any $t\in T^{*}\cup T^{(3)*}$, if $t$ appears $a$ times in $T^{*}\cup T^{(3)*}$, then $t$ appears $a+2$ times in $TT^{(2)}$. By Equation~(\ref{eq7}) again, we have $4n\ge(3-2)(4n-2(M-2))+(M-2)2$. This implies that $|X_3|=4n-2(M-2)$, $|X_M|=2$ and $|X_i|=0$ for $4\le i\le M-1$. From Equation (\ref{eq8}), we obtain
\begin{align*}
2n^2=&4n-2(M-2)+2\frac{(M-1)(M-2)}{2}\\
=&4n+(M-2)(M-3)\\
\le&4n+n(n+1).
\end{align*}
This is possible only for $n=5$. This finishes the proof of claim.

 By Lemma~\ref{lemma2}, we have
\begin{equation}\label{eq5}
\begin{split}
\sum_{i=3}^{M}(i-2)|X_i|=&\sum_{i=2}^{M}i|X_i|-2\sum_{i=2}^{M}|X_i|\\
=&4n^2+4n+1-2(2n^2+1)\\
=&4n-1,
\end{split}
\end{equation}
and
\begin{equation}\label{eq6}
\begin{split}
2n^2+\beta=&\sum_{i=3}^{M}\frac{(i-1)(i-2)}{2}|X_i|\\
\le&\frac{M-1}{2}\sum_{i=3}^{M}(i-2)|X_i|\\
=&\frac{M-1}{2}(4n-1).
\end{split}
\end{equation}
This leads to $M\ge\frac{4n^2+2\beta}{4n-1}+1$. Since $M\in\mathbb{Z}$, then $M\ge n+2$.

{\bf{Case 1: $|X_{M}|\ge2$.}}

Note that for any $t\in T$, $t$ appears at least $3$ times in $TT^{(2)}$.
From Equation~(\ref{eq5}), we obtain
\begin{align*}
4n-1=&\sum_{i=3}^{M}(i-2)|X_i|\\
\ge&(3-2)(2n+1-|X_M|)+n|X_M|\\
=&2n+1+(n-1)|X_M|.
\end{align*}
This leads to $|X_{n+2}|=2$, $|X_{3}|=2n-1$, $e\in X_3$, $\beta=1$ and $|X_i|=0$ for all $4\le i\le n+1$. By Lemma~\ref{lemma2}, we have
\begin{align*}
2n^2+1=&\sum_{i=3}^{M}\frac{(i-1)(i-2)}{2}|X_i|\\
=&|X_3|+\frac{(n+1)n}{2}|X_{n+2}|\\
=&2n-1+n^2+n\\
=&n^2+3n-1.
\end{align*}
This implies that $n^2-3n+2=0$, which is a contradiction.

{\bf{Case 2: $|X_M|=1$.}}

For this case, we have $X_M=\{e\}$ and $M=2\beta+1$.
Note that for any $t\in T^{*}\cup T^{(3)*}$, if $t$ appears $a$ times in $T^{*}\cup T^{(3)*}$, then $t$ appears $a+2$ times in $TT^{(2)}$. By Equation~(\ref{eq5}), we have
\begin{align*}
4n-1=&\sum_{i=3}^{M}(i-2)|X_i|\\
\ge&(3-2)(4n+1-M)+(M-2)\\
=&4n-1.
\end{align*}
This implies that $|X_3|=4n+1-M$ and $|X_i|=0$ for all $4\le i\le M-1$. By Lemma~\ref{lemma2}, we have
\begin{align*}
2n^2+\frac{M-1}{2}=&\sum_{i=3}^{M}\frac{(i-1)(i-2)}{2}|X_i|\\
=&|X_3|+\frac{(M-1)(M-2)}{2}\\
=&4n+1-M+\frac{(M-1)(M-2)}{2}.
\end{align*}
This leads to $M^2-6M-(4n^2-8n-5)=0$, then $M=2n+1$, and so $T=T^{(2)}$. Hence $(T^{*})^2=2G-T^{*}+(2n-2)e$ and $T^{*}$ is a $(v,k,\lambda,\mu)$ regular partial difference set with $v=2n^2+1$, $k=2n$, $\lambda=1$ and $\mu=2$. By \cite[Theorem 6.1]{M1994}, this is only possible for $n=11$.
\end{proof}
\subsection{$n\equiv1\pmod{3}$}
By Equation (\ref{eq4}), we obtain
\begin{align*}
T^4=&(8n^2+8n+2)G+(2n-2)T^2+T^2T^{(2)}\\
=&(8n^2+8n+2)G+(2n-2)(2G+T^{(2)}+(2n-2)e)+(2G+T^{(2)}+(2n-2)e)T^{(2)}\\
=&(8n^2+16n)G+(4n-4)T^{(2)}+(4n^2-8n+4)e+(T^{(2)})^2\\
=&(8n^2+16n+2)G+(4n-4)T^{(2)}+T^{(4)}+(4n^2-6n+2)e.
\end{align*}
This implies
\begin{align}
TT^{(3)}\equiv2G+T^{(4)}\pmod{3}.\label{eq11}
\end{align}
We write $TT^{(3)}=\sum_{i=0}^{L}iZ_i$, where $Z_i\ (i=0,1,\dots,L)$ form a partition of group $G$, i.e., $G=\sum_{i=0}^{L}Z_i$. Then it is easy to get the following equations.
\begin{align}
  &4n^2+4n+1=\sum_{i=1}^{L}i|Z_i|,\label{eq9}\\
  &2n^2+1=\sum_{i=0}^{L}|Z_i|.\label{eq10}
\end{align}
By Equation~(\ref{eq11}), we see that for any $t^4\in T^{(4)}$, $t^4=t\cdot t^3$, so $t^4$ appears at least 3 times in $TT^{(3)}$, and for any $g\in G\backslash T^{(4)}$, $g$ appears at least 2 times in $TT^{(3)}$. Hence $|Z_0|=|Z_1|=0$. Moreover, we have
\begin{align}
&\sum_{i\ge1}|Z_{3i}|=2n+1,\label{eq12}\\
&\sum_{i\ge0}|Z_{3i+2}|=2n^2-2n.\label{eq13}
\end{align}

\begin{lemma}\label{lemma3}
For any $g\in G$, $g$ appears at most 2 times in $T^{(3)}$.
\end{lemma}
\begin{proof}
 Assume to the contrary, suppose that $g$ appears at least $3$ times in $T^{(3)}$, then for any $t\in T$, $g\cdot t$ appears at least 3 times in $TT^{(3)}$. By Lemma~\ref{lemma1}, for any $g'\in T^{(3)}$, $g'\ne g$, we have $|gT\cap g'T|\ge2$. This implies that the elements in $gT$ appear at least $3(2n+1)+(2n-2)2$ times in $TT^{(3)}$. By Equation~(\ref{eq9}), we obtain
 \begin{align*}
&4n^2+4n+1\\
 =&\sum_{i=1}^{L}i|Z_i|\\
 \ge&2(2n^2-2n)+3(2n+1)+2(2n-2)\\
 =&4n^2+6n-1,
 \end{align*}
 which is a contradiction.
\end{proof}
Since $e$ appears odd times in $T^{(3)}$, then $e$ appears once in $T^{(3)}$.
\begin{lemma}\label{lemma4}
  For any $e\ne g\in G$, $g$ appears at most $n+2$ times in $TT^{(3)}$.
\end{lemma}
\begin{proof}
  If there exists $g\ne e$ such that $g$ appears at least $n+3$ times in $TT^{(3)}$, then $g^{-1}$ also appears at least $n+3$ times in $TT^{(3)}$. By Equation~(\ref{eq9}), we obtain
 \begin{align*}
&4n^2+4n+1\\
 =&\sum_{i=1}^{L}i|Z_i|\\
 \ge&2(2n^2-2n)+3(2n+1)+2(n+3)-3\cdot2\\
 =&4n^2+4n+3,
 \end{align*}
 which is a contradiction.
\end{proof}

\begin{lemma}\label{lemma7}
If $n\ge7$, then $|\text{supp}(T^{(3)})|=2n+1.$
\end{lemma}
\begin{proof}
If there exists $g\in G$, $g\ne e$ such that $g$ appears twice in $T^{(3)}$, then $g^{-1}$ also appears twice in $T^{(3)}$. By Lemma~\ref{lemma1}, for any $g'\in T^{(3)}$, $g'\ne g,g^{-1}$, we have $|gT\cap g'T|\ge2$ and $|g^{-1}T\cap g'T|\ge2$. By Lemma~\ref{lemma4}, we have that the elements in $gT\cup g^{-1}T$ appear at least $2(2n+1)+2(2n+1)+2(2n+1-4)+(2n+1-4-(n+2-2))=13n-5$ times in $TT^{(3)}$. By Equation~(\ref{eq9}), we have
 \begin{align*}
&4n^2+4n+1\\
 =&\sum_{i=1}^{L}i|Z_i|\\
 \ge&2(2n^2+1-|gT\cup g^{-1}T|)+13n-5\\
 \ge&2(2n^2+1-4n)+13n-5\\
 =&4n^2+5n-3,
 \end{align*}
this is possible only for $n=4$.
\end{proof}

With a similar discussion as Lemma~\ref{lemma2}, we have the following lemma.
\begin{lemma}\label{lemma5}
If $n\ge7$, then we have $\sum_{i=1}^{L}|Z_i|=1-\delta+\sum_{i=3}^{L}\frac{(i-1)(i-2)}{2}|Z_i|$, where $2\delta+1=|T\cap T^{(3)}|$.
\end{lemma}

\begin{proposition}\label{prop2}
  Theorem~\ref{mainthm} holds for $n\equiv1\pmod{3}$ except $n=4$.
\end{proposition}
\begin{proof}
By Equations (\ref{eq9}), (\ref{eq10}), (\ref{eq12}) and (\ref{eq13}), we get
\begin{align}
&\sum_{i=3}^{L}(i-2)|Z_i|=\sum_{i=2}^{L}i|Z_i|-2\sum_{i=2}^{L}|Z_i|=4n-1,\\
&\sum_{i\ge2}(3i-3)|Z_{3i}|+\sum_{i\ge1}3i|X_{3i+2}|=\sum_{i=3}^{L}(i-2)|Z_i|-\sum_{i\ge1}|Z_{3i}|=2n-2.\label{eq14}
\end{align}
From Lemma~\ref{lemma5}, we obtain
\begin{align*}
2n^2+\delta-4n+1=&\sum_{i=3}^{L}\frac{(i-1)(i-2)}{2}|Z_i|-\sum_{i=3}^{L}(i-2)|Z_i|\\
=&\sum_{i=3}^{L}\frac{(i-3)(i-2)}{2}|Z_i|\\
\le&\frac{L-2}{2}\sum_{i=3}^{L}(i-3)|Z_i|\\
\le&\frac{L-2}{2}(2n-2).
\end{align*}
This leads to $L\ge2n$. From Equation (\ref{eq14}), we have $|Z_L|=1$ and then $Z_L=\{e\}$. Equation (\ref{eq12}) implies that $L\equiv0\pmod{3}$, so $L=2n+1$. Hence $T=T^{(3)}$. Comparing Equation (\ref{eq11}) with Condition (4) in Theorem~\ref{groupring}, we obtain $T^{(2)}=T^{(4)}$, which contradicts to Lemmas~\ref{lemma6} and \ref{lemma7}.
\end{proof}

\subsection{$n\equiv0\pmod{3}$}
Since $\gcd(3,2n^2+1)=1$, then $|\text{supp}(T^{(3)})|=2n+1$. So $\beta=0$, i.e., $|T\cap T^{(2)}|=1$. Hence $e\in X_1$. By Equation~(\ref{eq4}), we obtain
\begin{align*}
T^{(2)}T\equiv T^{(3)}+G+2T\pmod{3}.
\end{align*}
This implies that for any $g\in G\backslash(T\cup T^{(3)})$, $g$ appears at least once in $T^{(2)}T$, and for any $t\in T$, $t=t^{-1}t^2\in T^{(2)}T$, $t^3=tt^2\in T^{(2)}T$. Hence $|X_0|=0$. Moreover,
\begin{align}
&\sum_{i\ge1}|X_{3i}|=a,\label{eq16}\\
&\sum_{i\ge0}|X_{3i+2}|=a,\\
&\sum_{i\ge0}|X_{3i+1}|=2n^2+1-2a,\label{eq15}
\end{align}
where $a=|T\backslash(T\cap T^{(3)})|=|T^{(3)}\backslash(T\cap T^{(3)})|$. Then $a\le2n$. By Lemma~\ref{lemma2}, we have
\begin{align*}
\sum_{i\ge1}3i(|X_{3i}|+|X_{3i+1}|+|X_{3i+2}|)=&\sum_{i=1}^{M}i|X_i|-\sum_{i\ge0}|X_{3i+1}|-2\sum_{i\ge0}|X_{3i+2}|\\
=&(4n^2+4n+1)-(2n^2+1-2a)-2a\\
=&2n^2+4n,
\end{align*}
and
\begin{align*}
2n^2=&\sum_{i=3}^{M}\frac{(i-1)(i-2)}{2}|X_i|\\
\ge&-2|X_3|+\sum_{i\ge1}3i(|X_{3i}|+|X_{3i+1}|+|X_{3i+2}|)\\
=&-2|X_3|+2n^2+4n.
\end{align*}
This leads to $|X_3|\ge2n$. Combining with Equations (\ref{eq16})-(\ref{eq15}) and $a\le2n$, we have
\begin{align*}
|X_1|=\frac{4}{3}n^2-\frac{10}{3}n+1,\ |X_2|=2n,\ |X_3|=2n,\ |X_4|=\frac{2}{3}n^2-\frac{2}{3}n\text{ and }|X_i|=0\text{ for all }i\ge5.
\end{align*}
Moreover
\begin{align*}
T^{(2)}T=X_1+2T^{(3)*}+3T^{*}+4X_4.
\end{align*}

In order to get a contradiction, we need to investigate $TT^{(4)}$. By Theorem~\ref{groupring}, we have
\begin{align*}
(T^{(2)})^2=&(2G+(2n-2)e-T^2)^2\\
=&(8n^2+8n-4)G+(4n^2-8n+4)e-(4G+(4n-4)e)T^2+T^4\\
=&(8n^2+8n-4)G+(4n^2-8n+4)e-(4G+(4n-4)e)(2G+T^{(2)}+(2n-2)e)+T^4\\
=&(-8n^2-16n)G+(-4n^2+8n-4)e-(4n-4)T^{(2)}+T^4.
\end{align*}
On the other hand, $T^{(2)}$ also satisfies conditions in Theorem~\ref{groupring}, then
\[(T^{(2)})^2=2G+T^{(4)}+(2n-2)e.\]
Combining above two equations, we obtain
 \[T^4=(8n^2+16n+2)G+(4n^2-6n+2)e+T^{(4)}+(4n-4)T^{(2)}.\]
  Multiplying $T$ to both sides, we can get
\begin{align}\label{eq17}
TT^{(4)}=T^{5}-(16n^3+40n^2+20n+2)G-(4n^2-6n+2)T-(4n-4)TT^{(2)}.
\end{align}
Writing $TT^{(4)}$ as
\begin{align*}
TT^{(4)}=\sum_{i=0}^{N}iY_i,
\end{align*}
where $N\le2n+1$, and $Y_i$, $i=0,1,\dots,N$ form a partition of group $H$, i.e.,
\begin{align*}
H=\sum_{i=0}^{N}Y_i.
\end{align*}
Then we have the following lemma.
\begin{lemma}\label{lemma8}
  \begin{enumerate}
    \item [(1)] $\sum_{i=0}^{N}|Y_i|=2n^2+1$,
    \item [(2)] $\sum_{i=1}^{N}i|Y_i|=4n^2+4n+1,$
    \item [(3)] $\sum_{i=1}^{N}|Y_i|=1-\gamma+\sum_{i\ge3}|Y_i|\frac{(i-1)(i-2)}{2}$, where $|T^{*}\cap T^{(4)*}|=2\gamma$.
  \end{enumerate}
\end{lemma}
\begin{proof}
  The proof is similar to that for Lemma~\ref{lemma2}.
\end{proof}

\begin{proposition}\label{prop3}
  Theorem~\ref{mainthm} holds for $n\equiv6\pmod{15}.$
\end{proposition}
\begin{proof}
 By Equation~(\ref{eq17}), we obtain
 \[TT^{(4)}\equiv T^{(5)}+2G\pmod{5}.\]
 This implies
 \begin{align*}
 &\sum_{i\ge0}|Y_{5i+3}|=2n+1,\\
 &\sum_{i\ge0}|Y_{5i+2}|=2n^2-2n,\\
 &|Y_i|=0\text{ for all }i\equiv0,1,4\pmod{5}.
 \end{align*}
 By Lemma~\ref{lemma8}, we have
 \begin{align}
 &\sum_{i\ge3}(i-2)|Y_i|=\sum_{i=1}^{N}i|Y_i|-2\sum_{i=1}^{N}|Y_i|=4n-1,\\
 &\sum_{i\ge1}5i(|Y_{5i+2}|+|Y_{5i+3}|)=\sum_{i=1}^{N}i|Y_i|-2\sum_{i\ge0}|Y_{5i+2}|-3\sum_{i\ge0}|Y_{5i+3}|=2n-2,\label{eq18}
 \end{align}
 and
 \begin{align*}
 2n^2+\gamma-4n+1=&\sum_{i=3}^{N}\frac{(i-1)(i-2)}{2}|Y_i|-\sum_{i\ge3}(i-2)|Y_i|\\
 =&\sum_{i=3}^{N}\frac{(i-3)(i-2)}{2}|Y_i|\\
 \le&\frac{N-2}{2}\sum_{i=3}^{N}(i-3)|Y_i|\\
 \le&\frac{N-2}{2}\sum_{i\ge1}5i(|Y_{5i+2}|+|Y_{5i+3}|)\\
 =&\frac{N-2}{2}(2n-2).
 \end{align*}
 Since $N\in\mathbb{Z}$, then $N\ge2n$. Equation~(\ref{eq18}) implies that $|Y_N|=1$. Since $Y_N=Y_{N}^{(-1)}$, then $Y_N=\{e\}$. Thus $N$ is odd, so $N=2n+1$. Hence $T=T^{(4)}$. By Equation (\ref{eq17}), we obtain
 \[T^{2}=T^{5}-(16n^3+40n^2+20n+2)G-(4n^2-6n+2)T-(4n-4)TT^{(2)}.\]
 Multiplying $T$ to both sides, we obtain
 \[T^{3}=T^{6}-(2n+1)(16n^3+40n^2+20n+2)G-(4n^2-6n+2)T^2-(4n-4)T^2T^{(2)}.\]
 Taking modulo 3, we have
 \begin{align*}
 T^{(3)}\equiv&(T^{(3)})^2+G+T^2+T^2T^{(2)}\pmod{3}\\
 \equiv& (2G+T^{(6)}+e)+G+(2G+T^{(2)}+e)+(2G+T^{(2)}+e)T^{(2)}\pmod{3}\\
 \equiv &2G+T^{(6)}+T^{(2)}+2e+2G+(T^{(2)})^2+T^{(2)}\pmod{3}\\
 \equiv &G+T^{(6)}+2T^{(2)}+2e+(2G+T^{(4)}+e)\pmod{3}\\
 \equiv &T^{(6)}+2T^{(2)}+T\pmod{3}.
 \end{align*}
Since $T\cap T^{(2)}=\{e\}$, then this is only possible for $T=T^{(3)}$ and $T^{(2)}=T^{(6)}$, which contradicts to $|T\cap T^{(3)}|=1$.
\end{proof}
\begin{proposition}\label{prop4}
  Theorem~\ref{mainthm} holds for $n\equiv3\pmod{15}$ except for $n=3$.
\end{proposition}
\begin{proof}
   By Equation~(\ref{eq17}), we obtain
\begin{align*}
TT^{(4)}\equiv& T^{(5)}+G+2TT^{(2)}\pmod{5}\\
\equiv &T^{(5)}+G+2X_1+4T^{(3)*}+T^{*}+3X_4\pmod{5}\\
\equiv&4e+T^{(5)*}+3X_{1}^{*}+2T^{*}+4X_4\pmod{5}.
\end{align*}
This implies that
\begin{align*}
&\sum_{i\ge0}|Y_{5i}|=|T^{(3)*}\backslash T^{(5)*}|+|T^{(5)*}\cap X_4|,\\
&\sum_{i\ge0}|Y_{5i+1}|=|T^{(5)*}\cap T^{(3)*}|,\\
&\sum_{i\ge0}|Y_{5i+2}|=|T^{*}\backslash T^{(5)*}|,\\
&\sum_{i\ge0}|Y_{5i+3}|=|X_{1}^{*}\backslash T^{(5)*}|+|T^{(5)*}\cap T^{*}|,\\
&\sum_{i\ge0}|Y_{5i+4}|=|X_4\backslash T^{(5)*}|+1+|T^{(5)*}\cap X_{1}^{*}|.
\end{align*}
Note that for any $g\in T^{(5)*}$, $g$ appears at least once in $TT^{(4)}$. By Lemma~\ref{lemma8}, we have
\begin{align*}
4n^2+4n+1=&\sum_{i=1}^{N}i|Y_i|\\
\ge&\sum_{i\ge0}|Y_{5i+1}|+2\sum_{i\ge0}|Y_{5i+2}|+3\sum_{i\ge0}|Y_{5i+3}|+4\sum_{i\ge0}|Y_{5i+4}|+5\sum_{i\ge1}|Y_{5i}|\\
\ge&2|T^{*}|+3|X_{1}^{*}|+4(|X_4|+1)+|T^{(5)*}|\\
=&2\cdot2n+3(\frac{4}{3}n^2-\frac{10}{3}n)+4(\frac{2}{3}n^2-\frac{2}{3}n+1)+2n.
\end{align*}
This leads to $8n^2-32n+9\le0$, which is only possible for $n=3$.
\end{proof}

\begin{proposition}\label{prop5}
  Theorem~\ref{mainthm} holds for $n\equiv0\pmod{15}$.
\end{proposition}
\begin{proof}
   By Equation~(\ref{eq17}), we obtain
\begin{align*}
TT^{(4)}\equiv& T^{(5)}+3G+3T+4TT^{(2)}\pmod{5}\\
\equiv &T^{(5)}+3G+3T++4X_1+3T^{(3)*}+2T^{*}+X_4\pmod{5}\\
\equiv&e+T^{(5)*}+2X_{1}^{*}+T^{(3)*}+3T^{*}+4X_4\pmod{5}.
\end{align*}
This implies that
\begin{align*}
&\sum_{i\ge0}|Y_{5i}|=|T^{(5)*}\cap X_4|,\\
&\sum_{i\ge0}|Y_{5i+1}|=|T^{(3)*}\backslash T^{(5)*}|+1,\\
&\sum_{i\ge0}|Y_{5i+2}|=|T^{(5)*}\cap T^{(3)*}|+|X_{1}^{*}\backslash T^{(5)*}|,\\
&\sum_{i\ge0}|Y_{5i+3}|=|T^{*}\backslash T^{(5)*}|+|T^{(5)*}\cap X_{1}^{*}|,\\
&\sum_{i\ge0}|Y_{5i+4}|=|T^{(5)*}\cap T^{*}|+|X_4\backslash T^{(5)*}|.
\end{align*}
Note that for any $g\in T^{(5)*}$, $g$ appears at least once in $TT^{(4)}$. By Lemma~\ref{lemma8}, we have
\begin{align*}
4n^2+4n+1=&\sum_{i=1}^{N}i|Y_i|\\
\ge&\sum_{i\ge0}|Y_{5i+1}|+2\sum_{i\ge0}|Y_{5i+2}|+3\sum_{i\ge0}|Y_{5i+3}|+4\sum_{i\ge0}|Y_{5i+4}|+5\sum_{i\ge1}|Y_{5i}|\\
\ge&|T^{(3)*}|+1+2|X_{1}^{*}|+3|T^{*}|+4|X_4|+|T^{(5)*}|\\
=&2n+1+2(\frac{4}{3}n^2-\frac{10}{3}n)+3\cdot 2n+4(\frac{2}{3}n^2-\frac{2}{3}n)+2n.
\end{align*}
This leads to $4n^2-10n\le0$, which is impossible.
\end{proof}
In order to solve the cases $n\equiv12\pmod{15}$ and $n\equiv9\pmod{15}$, we need to study $TT^{(5)}$.
By Equation~(\ref{eq17}), we have
\begin{equation}\label{eq19}
\begin{split}
T^6=&(2n+1)(16n^3+40n^2+20n+2)G+T^2T^{(4)}+(4n^2-6n+2)T^2+(4n-4)T^2T^{(2)}\\
=&(32n^4+96n^3+80n^2+24n+2)G+(2G+T^{(2)}+(2n-2)e)T^{(4)}+(4n^2-6n+2)(2G+T^{(2)}\\
&+(2n-2)e)+(4n-4)(2G+T^{(2)}+(2n-2)e)T^{(2)}\\
=&(32n^4+96n^3+104n^2+8n)G+T^{(2)}T^{(4)}+(2n-2)T^{(4)}+(12n^2-22n+10)T^{(2)}+(8n^3-20n^2+16n-4)e\\
&+(4n-4)(2G+T^{(4)}+(2n-2)e)\\
=&(32n^4+96n^3+104n^2+16n-8)G+(6n-6)T^{(4)}+(12n^2-22n+10)T^{(2)}+T^{(2)}T^{(4)}+(8n^3-12n^2+4)e.
\end{split}
\end{equation}

\begin{proposition}\label{prop6}
  Theorem~\ref{mainthm} holds for $n\equiv12\pmod{15}$.
\end{proposition}
\begin{proof}
  By Equation~(\ref{eq19}), we have
  \begin{equation}\label{eq20}
  \begin{split}
  TT^{(5)}\equiv& T^{(4)}+4T^{(2)}+T^{(2)}T^{(4)}\pmod{5}\\
  \equiv& e+T^{(4)*}+X_{1}^{(2)*}+2T^{(6)*}+2T^{(2)*}+4X_{4}^{(2)}\pmod{5}.
  \end{split}
  \end{equation}
  Writing $TT^{(5)}$ as
  \[TT^{(5)}=\sum_{i=0}^{L}iZ_i,\]
  where $Z_i$, $(i=0,1,\dots,L)$ form a partition of group $G$. Let
  \[|T^{(4)*}\cap X_{1}^{(2)*}|=a_1,\ |T^{(4)*}\cap T^{(6)*}|=a_2,\ |T^{(4)*}\cap X_{4}^{(2)}|=a_3.\]
  Then $a_1+a_2+a_3=2n.$ From Equation~(\ref{eq20}), we get
  \begin{align*}
  &\sum_{i\ge0}|Z_{5i}|=a_3,\\
  &\sum_{i\ge0}|Z_{5i+1}|=\frac{4}{3}n^2-\frac{10}{3}n+1-a_1,\\
  &\sum_{i\ge0}|Z_{5i+2}|=4n-a_2+a_1,\\
  &\sum_{i\ge0}|Z_{5i+3}|=a_2,\\
  &\sum_{i\ge0}|Z_{5i+4}|=\frac{2}{3}n^2-\frac{2}{3}n-a_3.
  \end{align*}
  Note that for any $t\in T$, $t^4=t^{-1}\cdot t^5$, then $|Z_0|=0$.
We can compute to get that
  \begin{align*}
  4n^2+4n+1=&\sum_{i=1}^{L}i|Z_i|\\
  \ge&\sum_{i\ge0}|Z_{5i+1}|+2\sum_{i\ge0}|Z_{5i+2}|+3\sum_{i\ge0}|Z_{5i+3}|+4\sum_{i\ge0}|Z_{5i+4}|+5\sum_{i\ge1}|Z_{5i}|\\
  =&\frac{4}{3}n^2-\frac{10}{3}n+1-a_1+2(4n-a_2+a_1)+3a_2+4(\frac{2}{3}n^2-\frac{2}{3}n-a_3)+5a_3\\
  =&4n^2+2n+1+a_1+a_2+a_3.
  \end{align*}
  This leads to $|Z_1|=\frac{4}{3}n^2-\frac{10}{3}n+1-a_1$, $|Z_2|=4n-a_2+a_1$, $|Z_3|=a_2$, $|Z_4|=\frac{2}{3}n^2-\frac{2}{3}n-a_3$, $|Z_5|=a_3$ and $|Z_i|=0$ for all $i\ge6$. We can also get $|T\cap T^{(5)}|=1$. With a similar discussion as Lemma~\ref{lemma2}, we have
  \begin{align*}
  2n^2+1=&\sum_{i=1}^{L}|Z_i|\\
  =&1+\sum_{i=3}^{L}\frac{(i-1)(i-2)}{2}|Z_i|\\
  =&1+a_2+3(\frac{2}{3}n^2-\frac{2}{3}n-a_3)+6a_3.
  \end{align*}
  This leads to $a_2+3a_3=2n$. Recall that $T^{(2)}T^{(4)}=X_{1}^{(2)}+2T^{(6)*}+3T^{(2)*}+4X_{4}^{(2)}$. On one hand, the elements in  $T^{(4)*}$ appear $a_1+2a_2+4a_3=(a_1+a_2+a_3)+a_2+3a_3=4n$ times in $T^{(2)}T^{(4)}$. On the other hand, since $|T^{(4)*}\cap t^2T^{(4)*}|=|T^{(4)}\cap t^2T^{(4)}|\ge2$, then the elements in  $T^{(4)*}$ appear at least $2n+2\cdot2n=6n$ times in $T^{(2)}T^{(4)}$, which is a contradiction.
\end{proof}
\begin{proposition}\label{prop7}
  Theorem~\ref{mainthm} holds for $n\equiv9\pmod{15}$.
\end{proposition}
\begin{proof}
    By Equation~(\ref{eq19}), we have
  \begin{equation}\label{eq21}
  \begin{split}
  TT^{(5)}\equiv& G+3T^{(4)}+4T^{(2)}+T^{(2)}T^{(4)}+4e\pmod{5}\\
  \equiv& 3e+3T^{(4)*}+2X_{1}^{(2)*}+3T^{(6)*}+3T^{(2)*}\pmod{5}.
  \end{split}
  \end{equation}
  Writing $TT^{(5)}$ as
  \[TT^{(5)}=\sum_{i=0}^{L}iZ_i,\]
  where $Z_i$, $(i=0,1,\dots,L)$ form a partition of group $G$. Let
  \[|T^{(4)*}\cap X_{1}^{(2)*}|=a_1,\ |T^{(4)*}\cap T^{(6)*}|=a_2,\ |T^{(4)*}\cap X_{4}^{(2)}|=a_3.\]
  Then $a_1+a_2+a_3=2n.$ From Equation~(\ref{eq21}), we get
  \begin{align*}
  &\sum_{i\ge0}|Z_{5i}|=\frac{2}{3}n^2-\frac{2}{3}n-a_3+a_1,\\
  &\sum_{i\ge0}|Z_{5i+1}|=a_2,\\
  &\sum_{i\ge0}|Z_{5i+2}|=\frac{4}{3}n^2-\frac{10}{3}n-a_1,\\
  &\sum_{i\ge0}|Z_{5i+3}|=4n+1-a_2+a_3,\\
  &|Z_i|=0\text{ for }i\equiv4\pmod{5}.
  \end{align*}

  {\bf{Claim: $a_3\le\frac{4}{3}n.$}}

  Recall that $T^{(2)}T^{(4)}=X_{1}^{(2)}+2T^{(6)*}+3T^{(2)*}+4X_{4}^{(2)}$. If $|T^{(4)*}\cap X_{4}^{(2)}|=a_3>\frac{4}{3}n$, then there exist $t_1,t_2,t_3\in T$ such that $t_{1}^{4}=t^2t_{4}^{4}$, $t_{2}^{4}=t^2t_{5}^{4}$ and $t_{3}^{4}=t^2t_{6}^{4}$ for some $t,t_4,t_5,t_6\in T$. This implies that $t^2=t_{1}^{4}t_{4}^{-4}=t_{2}^{4}t_{5}^{-4}=t_{3}^{4}t_{6}^{-4}$, and so $t_1=t_{5}^{-1}=t_{6}^{-1}$, $t_2=t_{4}^{-1}=t_{6}^{-1}$, $t_3=t_{4}^{-1}=t_{5}^{-1}$, which is a contradiction. This completes the proof of claim.

  Note that for any $g\in T^{(4)*}\cap T^{(6)*}$, $g$ appears at least $6$ times in $TT^{(5)}$. Hence $|Z_1|=0$.
 With a similar discussion as Lemma~\ref{lemma2}, we have
 \begin{align*}
 &4n^2+4n+1=\sum_{i=1}^{L}i|Z_i|,\\
 &\sum_{i=1}^{L}|Z_i|=1-\frac{|T^{*}\cap T^{(5)*}|}{2}+\sum_{i=3}^{L}\frac{(i-1)(i-2)}{2}|Z_i|.
 \end{align*}
 Then we can get
   \begin{equation}\label{eq22}
   \begin{split}
   &4n^2+4n+2-\frac{|T^{*}\cap T^{(5)*}|}{2}+\sum_{i=6}^{L}\frac{i(i-5)}{2}|Z_i|\\
   =&\sum_{i=1}^{L}i|Z_i|+\sum_{i=1}^{L}|Z_i|-\sum_{i=3}^{L}\frac{(i-1)(i-2)}{2}|Z_i|+\sum_{i=6}^{L}\frac{i(i-5)}{2}|Z_i|\\
   =&2|Z_1|+3|Z_2|+3|Z_3|+2|Z_4|\\
   =&3|Z_2|+3|Z_3|\\
   \le&3(\frac{4}{3}n^2-\frac{10}{3}n-a_1+4n+1-a_2+a_3)\\
   =&4n^2+2n+3+3(a_3-a_1-a_2)\\
   \le&4n^2+4n+3.
   \end{split}
   \end{equation}
   By Equation (\ref{eq21}), we see that $|T\cap T^{(5)}|=3$ or $|T\cap T^{(5)}|\ge13$. If $|T\cap T^{(5)}|\ge13$, then $4n^2+4n+2-\frac{|T^{*}\cap T^{(5)*}|}{2}+\sum_{i=6}^{L}\frac{i(i-5)}{2}|Z_i|\ge4n^2+4n+2-6+\frac{13(13-5)}{2}=4n^2+4n+48$, which is a contradiction. Hence $|T\cap T^{(5)}|=3$. Then $|Z_i|=0$ for all $i\ge6$. From Equation~(\ref{eq22}), we obtain $4n^2+4n+1=3|Z_2|+3|Z_3|$. This implies $3\mid(4n^2+4n+1)$, which is a contradiction.
\end{proof}
\subsection{Some sporadic cases}
For $n=11$, Wei and Schwartz \cite{WS22} have constructed a lattice tiling of $\mathbb{Z}^n$ by $\mathcal{B}(11,2,1,1)$.
For $n=3,4,5$, the nonexistence of lattice tiling with $\mathcal{B}(n,2,1,1)$ follows from the following proposition.
\begin{proposition}\label{prop8}
  Suppose that $2n^2+1=mp$ where $p$ is a prime and $p>2n+1$. Define $a=\min\{k\in\mathbb{Z}_{\ge0}:\ p\mid(4^k-4n-2)\}$ and $b$ is the order of $4$ modulo $p$ (If there is no $k$ with $p\mid(4^k-4n-2)$, then let $a=\infty$). Assume that there is a lattice tiling of $\mathbb{Z}^n$ by $\mathcal{B}(n,2,1,1)$, then there exists at least one $\ell\in\{0,1,\dots,\lfloor\sqrt{\frac{m-1}{2}}\rfloor\}$ such that the equation
  \[a(x+1)+by=n-\ell\]
has nonnegative integer solutions.
\end{proposition}
\begin{proof}
The proof is similar as \cite{K17}, for readers convenience, we keep the proof.

Let $G$ be a finite abelian group (written additively) with order $2n^2+1$, and $0$ be its identity element.
	By Theorem~\ref{tilinggroup}, there exists $T=\{t_i: i=1,\dots, n\}\subseteq G$ such that
	\[ \{0\}, \{\pm t_i:  i=1,\dots, n\}, \{\pm t_i\pm t_j: 1\leq i<j\leq n\}  \]
	form a partition of $G$.
	
	Let $H$ be a subgroup of $G$ of index $p$. Let $\rho:G \rightarrow G/H$ be the canonical homomorphism and $a_i=\rho(t_i)$. Then the multisets
	\[\{0\}, \{\pm a_i:  i=1,\dots, n\}, \{\pm a_i\pm a_j: 1\leq i<j\leq n\}  \]
	form a partition of $mG/H$.
	
Let $k$ be an integer. By calculation,
	\begin{align*}
	&\sum_{i=1}^{n}\left((a_i^{2k} + (-a_i)^{2k} \right)+\sum_{1\leq i < j\leq n}\left( (a_i+a_j)^{2k} + (a_i-a_j)^{2k} + (-a_i+a_j)^{2k} +(-a_i-a_j)^{2k} \right)\\
	=& (-4^{k} + 4n + 2)S_{2k} + 2\sum_{t=1}^{k-1}\binom{2k}{2t}S_{2t}S_{2(k-t)}
	\end{align*}
	where $S_t = \sum_{i=1}^{n}a_i^t$. Since this is also the sum of the $2k$-th powers of every element in $mG/H$,
	\begin{equation}\label{eq:kim_main}
	(-4^{k} + 4n + 2)S_{2k} + 2\sum_{t=1}^{k-1}\binom{2k}{2t}S_{2t}S_{2(k-t)}=
	\begin{cases}
	0,  & p-1 \nmid 2k,\\
	-m, & p-1 \mid 2k.
	\end{cases}
	\end{equation}
	
	Let $a$ and $b$ be the least positive integers satisfying $p\mid (-4^a+4n+2)$ and $p\mid (4^b-1)$. Define
	\[ X = \{ax+by : x\geq 1, y\geq 0\}. \]
	We prove the following two claims by induction on $k$.
	
	\textbf{Claim 1:} If $1\le k < \frac{p-1}{2}$ is not in $X$, then $S_{2k}=0$.
	
	Suppose that $S_{2k}=0$ for each $k\le k_0-1$ that is not in $X$. Assume that $k_0\notin X$. As $X$ is closed under addition, for any $t$, at least one of $t$ and $k_0-t$ is not in $X$.
	
	For any integer $k$, if $p\mid (-4^k+4n+2)$, then $k$ must be of the form $a+by$ whence $k\in X$. This implies that $p\nmid (-4^{k_0}+4n+2)$. By \eqref{eq:kim_main} and the induction hypothesis,
	\[  0=(-4^{k_0} + 4n + 2)S_{2k_0} + 2\sum_{t=1}^{k_0-1}\binom{2k_{0}}{2t}S_{2t}S_{2(k_0-t)}=(-4^{k_0} + 4n + 2)S_{2k_0}.\]
	Thus $S_{2k_0}= 0$.
	
	Let $e_k$ be the elementary symmetric polynomials with respect to $a_1^2$, $a_2^2$, $\cdots$, $a_n^2$.
	
	\textbf{Claim 2}: If $1\le k\le n< \frac{p-1}{2}$ is not in $X$, then $e_k=0$.
	
	Suppose that $e_k=0$ for each $k\leq k_0-1$ not in $X$ and $k_0\notin X$. As $X$ is closed under addition, for each $0<t<k_0$, at least one of $t$ and $k_0-t$ is not in $X$. By Claim 1 and the inductive hypothesis, $e_t=0$ or $S_{2(k_0-t)}=0$.
	Together with Newton identities on $a_1^2,\dots, a_n^2$, we have
	\[ k_0e_{k_0} = e_{k_0-1}S_2 + \dots + (-1)^{i+1} e_{k_0-i}S_{2i}+\dots +(-1)^{k_0-1}S_{2k_0}=(-1)^{k_0-1}S_{2k_0}=0.\]
	Therefore $e_{k_0}=0$.
	
 Suppose that $0$ appears $\ell$ times in $\rho(T)$, and let $a_{i_1}=a_{i_2}=\dots=a_{i_\ell}=0$, then $-a_{i_j}=0$ and $\pm a_{i_j}\pm a_{i_k}=0$ for $1\le j<k\le \ell$. This means that the number of $0$ in multiset $\rho(G)$ is at least $1+2\ell+4\binom{\ell}{2}=2\ell^{2}+1$.
 As $0$ appears exactly $m$ times in $mG/H$, then $2\ell^{2}+1\le m$, i.e., $\ell\le\sqrt{\frac{m-1}{2}}$. Note that $e_{n-\ell}$ is the product of those nonzero $a_i^2$'s, whence $e_{n-\ell}\neq 0$. By Claim 2, $n-\ell$ is in $X$. This completes the proof.
\end{proof}

Theorem~\ref{mainthm} is a combination of Propositions \ref{prop1}, \ref{prop2}, and \ref{prop3}---\ref{prop8}.

\bibliographystyle{IEEEtranS}

\end{document}